\pdfoutput=1
\documentclass[10pt, a4paper, twoside]{amsart}
\title{Koszulness of Enveloping Algebras Associated to Generalized Yang-Baxter Equations}
\date{\today}
\author{Robert Laugwitz}
\address{Mathematical Institute, University of Oxford, Andrew Wiles Building, Radcliffe Observatory Quarter, Woodstock Road, Oxford, OX2 6GG}
\email{laugwitz@maths.ox.ac.uk}
\urladdr{http://www.maths.ox.ac.uk/people/profiles/robert.laugwitz}

\usepackage[
colorlinks=true,
linkcolor=black, % einfache interne Verknpfungen
anchorcolor=black,% Ankertext
citecolor=black, % Verweise auf Literaturverzeichniseintrge im Text
urlcolor=black, % Farbe der URLs
]{hyperref}
\usepackage{import}
\usepackage{amsmath}
\usepackage{mathabx}
\usepackage{lipsum}
\usepackage{multicol}
\usepackage{amsthm}
\usepackage[latin1]{inputenc}
\usepackage[english]{babel}
\usepackage[alphabetic, initials]{amsrefs}
\usepackage{amssymb}
\usepackage{fancyhdr}
\usepackage{graphicx}
\usepackage{dsfont}
\usepackage{amsfonts}
\usepackage{geometry}
\usepackage[small,nohug,heads=vee]{diagrams}
\diagramstyle[labelstyle=\scriptstyle]
\newarrow{Mono}{C}{-}{-}{-}{>}

\makeatletter
\def\imod#1{\allowbreak\mkern10mu({\operator@font mod}\,\,#1)}
\makeatother

\newcommand{\bigslant}[2]{{\raisebox{.2em}{$#1$}\left/\raisebox{-.2em}{$#2$}\right.}}

\makeatletter
\newcommand{\superimpose}[2]{%
  {\ooalign{$#1\@firstoftwo#2$\cr\hfil$#1\@secondoftwo#2$\hfil\cr}}}
\makeatother

\newcommand{\hooklongrightarrow}{\lhook\joinrel\longrightarrow}
\newcommand{\twoheadlongrightarrow}{\relbar\joinrel\twoheadrightarrow}

\newcommand{\leftexpsub}[3]{{\vphantom{#3}}^{#1}_{#2}{#3}}

\newcommand{\lmod}[1]{#1\text{-}\mathbf{Mod}}

\newcommand{\un}[1]{\underline{#1}}
\newcommand{\uun}[1]{\underline{\underline{#1}}}
\newcommand{\ov}[1]{\overline{#1}}
\newcommand{\oov}[1]{\overline{\overline{#1}}}

\newcommand{\Drin}{\operatorname{Drin}}

\newcommand{\End}{\operatorname{End}}
\newcommand{\Ext}{\operatorname{Ext}}

\newcommand{\GL}{\operatorname{GL}}

\newcommand{\Heis}{\operatorname{Heis}}

\newcommand{\Hom}{\operatorname{Hom}}
\newcommand{\ide}{\operatorname{Id}}

\newcommand{\Ind}{\operatorname{Ind}}

\providecommand{\fr}[1]{\mathfrak{#1}}

\providecommand{\op}[1]{\operatorname{#1}}

\newcommand{\mC}{\mathds{C}}

\newcommand{\mA}{\mathds{A}}

\newcommand{\cB}{\mathcal{B}}
\newcommand{\cE}{\mathcal{E}}

\newcommand{\cL}{\mathcal{L}}

\newcommand{\cT}{\mathcal{T}}

\newcommand{\cS}{\mathcal{S}}

\newtheoremstyle{mystyle}% name of the style to be used
  {0.5cm}                   %Space above
  {0.5cm}                   %Space below
  {\normalfont}           %Body font
  {}                      %Indent amount (empty = no indent,
  {\itfont\bfseries}  %Thm head font
  {:}                     %Punctuation after thm head
  {0.3cm}              %Space after thm head: " " = normal interword
  {\thmname{#1}}
\newtheoremstyle{defstyle}% name of the style to be used
  {0.5cm}                   %Space above
  {0.5cm}                   %Space below
  {\normalfont}           %Body font
  {}     %Indent amount (empty = no indent,
  {\normalfont\bfseries}  %Thm head font
  {:}                     %Punctuation after thm head
  {0.3cm}              %Space after thm head: " " = normal interword
  {\thmname{#1}\thmnumber{ #2}\thmnote{ (#3)}}
\numberwithin{equation}{section}
\newtheorem{theorem}{Theorem}[subsection]

\newtheorem{proposition}[theorem]{Proposition}
\newtheorem{corollary}[theorem]{Corollary}
\newtheorem{lemma}[theorem]{Lemma}

\theoremstyle{definition}
\newtheorem{definition}[theorem]{Definition}

\newtheorem{algorithm}[theorem]{Algorithm}

\theoremstyle{remark}
\newtheorem{example}[theorem]{Example}
\newtheorem{remark}[theorem]{Remark}

\renewcommand{\sectionmark}[1]		% Schriftform fr \section = Kaptlchen
	{
	\markboth{\small\it \thesection{} #1}{}
	}

\begin{document}

\begin{abstract}
The universal enveloping algebra $U(\fr{tr}_n)$ of a Lie algebra associated to the classical Yang-Baxter equation was introduced in \cite{BEER} where it was shown to be Koszul. This algebra appears as the $A_{n-1}$ case in a general class of braided Hopf algebras in \cite{BB} for any complex reflection group. In this paper, we show that the algebras corresponding to the series $B_n$ and $D_n$, which are again universal enveloping algebras, are Koszul. We further show how results of \cite{BB} can be used to produce pairs of adjoint functors between categories of rational Cherednik algebra representations of different rank and type for the classical series of Coxeter groups.
\end{abstract}


\maketitle

%\tableofcontents

%%%%%%%%%%%%%%%%%%%%%%%%%%%%%%%%%%%%%%%%%%%%%%%%%%%%%%%%%%%%%%%%%%%%%%%%%%%%%%%%%%%%%%%%%%%%%%%%%%%%%%%%%%

\section{Introduction}

\subsection{Motivation}

The Yang-Baxter equations 
\begin{align}
R_{12}R_{13}R_{23}&=R_{23}R_{13}R_{12}
\end{align}
play a major role in the study of integrable system which are of importance in quantum field theory and statistical mechanics (see e.g. \cite{Jim}). The \emph{classical} Yang-Baxter equations
\begin{align}
[r_{12},r_{13}]+[r_{12},r_{23}]+[r_{13},r_{23}]&=0,
\end{align}
can be obtained as a classical limit (see \emph{loc.cit.}, Section~3) of these equations. There exist natural generalizations to $n$ indices
\begin{align}\label{classicalybe}
[r_{ij},r_{ik}]+[r_{ij},r_{jk}]+[r_{ik},r_{jk}]&=0,
\end{align}
for $1\leq i<j<k\leq n$. In \cite{BEER}, the Lie algebra $\fr{tr}_n$ with generators $r_{ij}$ subject to the relations (\ref{classicalybe}) and $[r_{ij},r_{kl}]=0$ for four distinct indices is considered. The corresponding discrete group such that the associated Malcev Lie algebras are $\fr{tr}_n$ are also studied in \cite{BEER}. These groups are the pure flat braid groups of \cite{Lee}.

It was shown in \cite{BEER} that the universal enveloping algebras $U(\fr{tr}_n)$ are Koszul. This result is in analogy with the Koszulness of the Drinfeld-Kohno Lie algebras, which have a similar presentation, with generators $r_{ij}$, and relations $[r_{ij},r_{kl}]=0$, $[r_{ij},r_{ik}]+[r_{ij},r_{jk}]=0$ (see e.g. \cite[3.10]{EHKR}).

In \cite{BB}, the algebra $U(\fr{tr}_n)$ is reinterpreted as a certain quadratic cover\footnote{The minimal quadratic cover of $\cB(Y_G)$ is the algebra $\cE_n$ of \cite{FK}. It is a quotient of $U(\fr{tr}_n)$.} of the Nichols algebra $\cB(Y_G)$ of the Yetter-Drinfeld module $Y_G$ which has a basis parametrized by reflections $(ij)\in S_n$. The module structure is (up to a cocycle) given by the conjugation action on reflections and the grading is given by the corresponding group element $(ij)$ of the generator. This description gives natural generalizations of $U(\fr{tr}_n)$ for any complex reflection group $G$ (see Section~\ref{sect1} for a brief summary of their construction). These algebras are in particular braided Hopf algebras in the category of YD-modules over $G$ and can be viewed as the universal enveloping algebras of certain Lie algebras, which are described in terms of generators and relations (see  Section \ref{generalizedbeer}). For explicit presentations in the cases of the classical series $D_n$ and $B_n$, see Section \ref{dncase}, \ref{bncase}. We will refer to these generalizations as \emph{BEER-algebras}.

The main reason for our interest in the algebras $U(\fr{yb}_G)$ lies in their applications to \emph{rational Cherednik algebras}. The rational Cherednik algebras $H_{t,c}(G)$ were introduced in \cite{EG} and are deformations of the algebra $\mC[T^*(V)]\rtimes kG$ for $t=0$, and $D(V)\rtimes kG$ for $t\neq 0$. The algebras $H_{t,c}(G)$ and their symplectic generalizations have since been at the core of an active area of current research (see e.g. \cite{EG,Rou,GS} and many more). One important tool in the study of these algebras is the \emph{Dunkl embedding}

\begin{minipage}{0.5\textwidth}
\begin{align*}
\Theta_c\colon H_{0,c}(G)&\hooklongrightarrow \mC[\fr{h}^*\times \fr{h}^{\op{reg}}]\rtimes G,&
\end{align*}
\end{minipage}
\begin{minipage}{0.5\textwidth}
\begin{align*}
\Theta_c\colon H_{1,c}(G)&\hooklongrightarrow D(\fr{h}^{\op{reg}})\rtimes G.&
\end{align*}
\end{minipage}

\noindent In \cite{BB} different embeddings were introduced, namely
\begin{align*}
M_c\colon H_{0,c}(G)&\hooklongrightarrow \Heis_{\mC G}(U(\fr{yb}_G)),&\\
M_c\colon H_{t,c}(G)&\hooklongrightarrow \Heis_{\mC G}(U(\fr{yb}_G))_{t'}.&
\end{align*}
The analogy is that the ring of differential operators is the Heisenberg double\footnote{The subscript $t'$ denotes a certain deformation parameter (cf. \cite{BB}).} of the ring of regular functors (for $\mA^n$). In some sense, the embedding into the braided Heisenberg double is more symmetric than the Dunkl embedding, which is the identity on $\fr{h}$ while it maps $y\in \fr{h}^*$ to the corresponding Dunkl operator $D_y$. In the  embeddings of \cite{BB}, elements of both $\fr{h}$ and its dual are mapped to similar expressions which are summations over all reflections $s\in \cS$. More concretely, in the case $t=0$,
\begin{align*}
D_y=\sum_{s\in \cS}{c_s(y,\alpha_s)\frac{1-s}{\alpha_s}}, \qquad M_c(x)=\sum_{s\in \cS}{c_s(\alpha_s^*,x)\un{s}},\qquad M_c(y)=\sum_{s\in \cS}{(y,\alpha_s)\ov{s}}.
\end{align*}

In this paper, we want to demonstrate one possible application of the embedding $M_c$. Namely, the construction of induction (right exact) and restriction (left exact) functors of modules over rational Cherednik algebras between different types (for the classical series $A_n$, $B_n$ and $C_n$). We focus on the case $t=0$, although similar constructions can be done in the case $t\neq 0$ as well.

\subsection{Summary}

The paper starts with a brief review of the construction of the generalizations of the algebras $U(\fr{tr}_n)$ due to \cite{BB} in Section~\ref{sect1}. We give an interpretation as universal enveloping algebras of a certain Lie algebra $\fr{yb}_G$ and explicit descriptions of the examples corresponding to the $D_n$ and $B_n$-series. We will also denote the corresponding Weyl groups by $A_{n-1}=S_n$, $B_n$, $D_n$.

In Section \ref{sect3}, we prove that the algebras $U(\fr{yb}_{G_n})$ are Koszul\footnote{It was pointed out by Y. Bazlov that Koszulness of these algebras has been conjectured by A.N. Kirillov in 2006 based on computational evidence.} for $G_n = D_n$ and $B_n$. This is done by explicitly constructing a PBW basis of the corresponding quadratic dual in both cases separately.

In the final Section \ref{sect2}, we consider maps between the braided Heisenberg doubles of $U(\fr{yb}_G)$. First, there are injective algebra morphisms
\begin{align*}
\phi_A^D\colon \Heis_{\mC S_n}(U(\fr{tr}_{n}))&\hooklongrightarrow \Heis_{\mC D_n}(U(\fr{yb}_{D_n})),&\\
\phi_A^B\colon \Heis_{\mC S_n}(U(\fr{tr}_{n}))&\hooklongrightarrow \Heis_{\mC B_n}(U(\fr{yb}_{B_n})),&
\end{align*}
but also injective algebra morphisms
\begin{align*}
\tau_n^{n}\colon\Heis_{\mC G_n}(U(\fr{yb}_{G_n}))&\hooklongrightarrow \Heis_{\mC G_{n}}(U(\fr{yb}_{G_{n}})),&
\end{align*}
where $G$ can be replaced by either $A$, $D$ or $B$.
These morphisms can be combined with the result of \cite{BB} that the rational Cherednik algebras $H_{0,c}(G)$ map to the braided Heisenberg doubles of $U(\fr{yb}_G)$. Hence we can use the diagrams\footnote{Diagrams are created using the package available at \url{http://www.paultaylor.eu/diagrams/}.}

\noindent\begin{minipage}{0.5\linewidth}
\begin{diagram}
&&\Heis_{\mC G_n}(U(\fr{yb}_{G_n}))&&\\
&\ruTo_{M_c}&&\luTo_{M_c}&\\
H_{0,c}(S_n)&&&&H_{0,c}(G_n).
\end{diagram}
\end{minipage}
\begin{minipage}{0.5\linewidth}
\begin{diagram}
&&\Heis_{\mC G_{n+1}}(U(\fr{yb}_{G_{n+1}}))&&\\
&\ruTo_{M_c}&&\luTo_{M_c}&\\
H_{0,c}(G_n)&&&&H_{0,c}(G_{n+1}),
\end{diagram}
\end{minipage}
\vspace{5pt}

\noindent where $G_n=S_n, D_{n}$ or $B_n$, to give the Heisenberg doubles a bimodule structure over different rational Cherednik algebras and use these to define functors 
\begin{align*}
\op{HInd}_{A}^{D}\colon \lmod{H_{0,c}(S_n)}&\longrightarrow \lmod{H_{0,c}(G_n)},&\\
\op{HInd}_n^{n+1}\colon \lmod{H_{0,c}(G_n)}&\longrightarrow \lmod{H_{0,c}(G_{n+1})},&\\
\op{HInd}_{D}^{A}\colon \lmod{H_{0,c}(G_n)}&\longrightarrow \lmod{H_{0,c}(S_n)},&\\
\op{HInd}_{n+1}^{n}\colon \lmod{H_{0,c}(G_{n+1})}&\longrightarrow \lmod{H_{0,c}(G_{n})},&
\end{align*}
as well as the corresponding restriction functors $\op{HRes}$ in the other directions. By the general tensor-Hom-adjunction, the induction and restriction functors form adjoint pairs. Hence all versions of the functors $\op{HInd}$ are right exact.

\subsection{Further Applications}

Another reason for this paper is the role $U(\fr{yb}_G)$ plays in the construction of categorical actions of the braided Drinfeld doubles of $U(\fr{yb}_G)$ over $\Drin(G)$ on the category of representations of the rational Cherednik algebras $H_{t,c}(G)$. The construction of the action uses a special case of a more general phenomenon that the monoidal category of modules over the braided Drinfeld double of a braided Hopf algebra acts on modules over the corresponding braided Heisenberg double (this generalization of an earlier result of \cite{Lu} is the content of \cite{Lau}). In \cite{Lau2}, we combine this general categorical action with the embeddings $M_c$ of \cite{BB} and obtain a categorical action of $\lmod{\Drin_{\Drin(G)}(U(\fr{yb}_G))}$ on $\lmod{H_{t,c}(G)}$.

\subsection{Acknowledgement}

First and foremost, I would like to thank my PhD advisor Dr Kobi Kremnizer for his support and guidance. I also like to thank Dr Yuri Bazlov, Prof Arkady Berenstein, Prof Pavel Etingof and Prof Iain Gordon for helpful conversations on the subject matter, and Dr Yuri Bazlov for helpful feedback on a first draft of this paper.

This research is supported by the EPSRC grant EP/I033343/1 \emph{Motivic Invariants and Categorification}\footnote{\url{http://gow.epsrc.ac.uk/NGBOViewGrant.aspx?GrantRef=EP/I033343/1}}.

%%%%%%%%%%%%%%%%%%%%%%%%%%%%%%%%%%%%%%%%%%%%%%%%%%%%%%%%%%%%%%%%%%%%%%%%%%%%%%%%%%%%%%%%%%%%%%%%%%%%%%%%%%

\section{BEER-Algebras for Complex Reflection Groups}\label{sect1}

\subsection{Notation}

Let $V$ be a finite-dimensional $\mC$-vector space, $G\leq \GL(V)$ a finitely-generated \emph{irreducible} complex reflection group with finite set of reflections
\begin{align*}
\cS&=\lbrace s\in G\mid \dim_{\mC}(1-s)=1\rbrace.&
\end{align*}
For each $s\in \cS$, $V=\ker(1-s)\oplus \op{im}(1+s)=V_1\oplus V_{\chi}$ as a decomposition of eigenspaces. We can choose vectors $\alpha_s\in V$, $\alpha_s^*\in V^*$ such that
\begin{align}
s(v)&=v-\langle \alpha_s^*,v\rangle \alpha_s, &\forall v\in V.
\end{align}
Here, $\alpha_s^*\otimes \alpha_s$ corresponds to $1-s\in \End(V)$ and is hence independent of choice. As $g\triangleright (1-s)=1-gsg^{-1}$, we have that $g\triangleright \alpha^*_s\otimes \alpha_s=\alpha^*_{gsg^{-1}}\otimes \alpha_{gsg^{-1}}$ and we can, following \cite{BB}, define $\lambda$ to be the function such that
\begin{align}
g\triangleright \alpha_s^*&=\lambda(g,s)\alpha_{gsg^{-1}}^*,&\forall g\in G, s\in \cS.
\end{align}
For a finite Coxeter group, $\lambda(g,s)$ can be chose to have values in $\lbrace \pm 1\rbrace$.

\subsection{Generalized BEER-Algebras}\label{generalizedbeer}

We briefly recall the construction of the quadratic algebras $U(\fr{yb}_G)$ associated to a complex reflection group $G$ from \cite[Chapter~7]{BB}.

\begin{definition}
Let $Y_G$ denote the Yetter-Drinfeld module over $kG$ generated by $\un{s}$, for $s\in \cS$, where
\begin{align}
g\triangleright \un{s}&=\lambda(g,s)\un{gsg^{-1}},&\\
\delta(\un{s})&=s\otimes \un{s}\in kG\otimes Y_G.&
\end{align}
For the braiding $\Psi\colon Y_G\otimes Y_G \to Y_G\otimes Y_G$, we consider the ideals
\begin{align*}
I^{\op{quad}}_G&:=\langle \ker(\ide_{Y_G\otimes Y_G}+\Psi)\rangle\triangleleft T(Y_G),&\\
I_G&:={I^{\Lambda,\op{quad}}_G}:=\langle I^{\op{quad}}\cap \Lambda Y_G\rangle\triangleleft T(Y_G).&
\end{align*}
We consider the following Hopf algebras in the category of YD-modules over $kG$,
\begin{align*}
\cB^{\op{quad}}(Y_G)&:=T(V)/{I^{\op{quad}}_G},&\\
\cB^{\op{quad}}_\Lambda(Y_G)&:=T(V)/{I_G}.&
\end{align*}
The algebra $\cB^{\op{quad}}_\Lambda(Y_G)$  will be referred to as the \emph{BEER-algebra for $G$}.
%We denote the generators of the dual YD-module $Y_G$ by $\ov{s}$ for $s\in \cS$.
\end{definition}

Note that the definition does not depend on the choices of $\alpha_s$, $\alpha_s^*$ as different choices will give isomorphic YD-modules (cf. \cite[Remark 7.16]{BB}). Moreover, $\lambda$ can be extended to a $\Drin(G)$-character by setting
\begin{align*}
g\otimes \delta_t\triangleright \un{s} &= \delta_{t,s}\lambda(g,s)\un{gsg^{-1}}.&
\end{align*}
This is equivalent to $\lambda(gh,s)=\lambda(g,hsh^{-1})\lambda(h,s)$.

\begin{remark}
The algebra $\cB^{\op{quad}}(Y_G)$ is a generalization of the algebra $\cE_n$ of \cite{FK} (which is the $A_{n-1}$-case) to arbitrary finite complex reflection groups.
\end{remark}

As for type $A$ in \cite{BEER}, we can consider a description as a universal enveloping algebra of a Lie algebra for the generalized BEER-algebras. Since $I_G\subseteq \Lambda V$, we can consider the quotient of the free Lie algebra $\cL(\cS)$ on generators $\un{s}$, for $s\in \cS$,
\begin{align*}
\fr{yb}_G&:=\bigslant{\cL(\cS)}{\langle I_G \rangle}.&
\end{align*}

\begin{example}$~$
\begin{enumerate}
\item[(i)] The cyclic group with two elements $C_2$ is a complex reflection group. For this group, $U(\fr{yb}_{C_2})=\mC[\un{s}]$, which is a Koszul algebra.
\item[(ii)] More generally, for $G=S_n$, the Lie algebra is $\fr{tr}_n$ from \cite{BEER}, for which $U(\fr{tr}_n)$ was shown to be Koszul.
\end{enumerate}
\end{example}

\begin{lemma}
There is an isomorphism of algebras
\begin{align*}
U(\fr{yb}_G)&\cong \cB^{\op{quad}}_\Lambda(Y_G),&
\end{align*}
for any complex reflection group $G$.
\end{lemma}
\begin{proof}
This can be seen by use of the universal properties of quotients of free Lie (respectively associative) algebras.
\end{proof}

Note that the algebra $U(\fr{yb}_G)$ carries a natural grading, where the generators $\un{s}$ have degree 1 and the relations are in degree 2, so the BEER-algebras are quadratic. Remark that the isomorphisms of $U(\fr{yb}_G)$ and $\cB^{\op{quad}}_\Lambda(Y_G)$ is one of algebras, \emph{not} of Hopf algebras. The coproduct of the latter is defined in the category of YD-modules over $G$. Hence $U(\fr{yb}_G)$ has two different Hopf algebra structures.

\subsection{The \texorpdfstring{$D_n$}{Dn}-Case}\label{dncase}

We will give explicit relations for the algebras $U(\fr{yb}_G)$ in the cases $G=D_n$, $B_n$ for which we will investigate Koszulness in Section~\ref{sect2}, starting with the $D_n$-case. The Weyl group $D_{n}$ of type $D_{n}$ is
\begin{align*}
D_{n}&:=C_2^{n-1}\rtimes S_{n},&
\end{align*}
acting on $\mC^{n}$, with reflections parametrized by the positive roots $e_i\pm e_j$ for $1\leq i<j\leq n$. The corresponding reflections are
\begin{align*}
s_{e_i-e_j}&=(ij)\in S_{n}\leq D_n,&i<j,\\
s_{e_i+e_j}&=s_is_j(ij),&i<j,
\end{align*}
where $s_i$ changes the sign of $e_i$. We denote the corresponding generators of $U(\fr{yb}_{D_{n}})$ by $\un{(ij)}$ for $s_{e_i-e_j}$ and $\un{\un{(ij)}}$ for $s_{e_i+e_j}$, where $1\leq i<j\leq n$. The character $\lambda$ will be most easily described by first remembering the order of the indices in $\sigma\triangleright(ij)=(\sigma(i),\sigma(j))$ and multiplying by $(-1)$ if the order of the indices of the transposition $(\sigma(i),\sigma(j))$ is reversed. That gives the module structure (for a transposition $\sigma=(k,l)$)
\begin{align}\label{dnmodulestructure}
\sigma \triangleright \un{(ij)}
&=\begin{cases}\un{(\sigma(i),\sigma(j))},& \text{if }\sigma(i)<\sigma(j),\\
-\un{(\sigma(j),\sigma(i))},& \text{if }\sigma(i)>\sigma(j),
\end{cases}\\
\sigma \triangleright \uun{(ij)}
&=\begin{cases}\uun{(\sigma(i),\sigma(j))},& \text{if }\sigma(i)<\sigma(j),\\
-\uun{(\sigma(j),\sigma(i))},& \text{if }\sigma(i)>\sigma(j),
\end{cases}\\
s_ks_l\sigma \triangleright \uun{(ij)}
&=\begin{cases}-\uun{(\sigma(i),\sigma(j))},& \text{if }\sigma=(ij),\\
\uun{(\sigma(i),\sigma(j))},& \text{if }\lbrace k,l\rbrace\cap\lbrace i,j\rbrace=\emptyset,\\
\un{(\sigma(i),\sigma(j))},& \text{if }\sigma(i)<\sigma(j) \text{ and precisely one index agrees},\\
-\un{(\sigma(j),\sigma(i))},& \text{if }\sigma(i)>\sigma(j) \text{ and precisely one index agrees},
\end{cases}\\
s_ks_l\sigma \triangleright \un{(ij)}
&=\begin{cases}-\un{(\sigma(i),\sigma(j))},& \text{if }\sigma=(ij),\\
\un{(\sigma(i),\sigma(j))},& \text{if }\lbrace k,l\rbrace\cap\lbrace i,j\rbrace=\emptyset,\\
\uun{(\sigma(i),\sigma(j))},& \text{if }\sigma(i)<\sigma(j) \text{ and precisely one index agrees},\\
-\uun{(\sigma(j),\sigma(i))},& \text{if }\sigma(i)>\sigma(j) \text{ and precisely one index agrees}.
\end{cases}
\end{align}

\begin{lemma}\label{iquaddn}
The ideal $I^{\op{quad}}_{D_{n}}$ is generated by the relations for $\lbrace i, j\rbrace \cap \lbrace k, l\rbrace=\emptyset$,

\noindent\begin{minipage}{0.5\textwidth}
\begin{align}
[\un{(ij)},\un{(kl)}]&=0,&\label{relduo1}\\
[\uun{(ij)},\uun{(kl)}]&=0,&\label{relduo2}\\
[\un{(ij)},\uun{(kl)}]&=0,&\label{relduo3}
\end{align}
\end{minipage}
\begin{minipage}{0.5\textwidth}
\begin{align}
\un{(ij)}^2&=0,&\label{relduo4}\\
\uun{(ij)}^2&=0,&\label{relduo5}\\
\un{(ij)}\uun{(ij)}+\uun{(ij)}\un{(ij)}&=0.\label{relduo6}&
\end{align}
\end{minipage}

\noindent And for $1\leq k<j<l\leq n$,

\noindent\begin{minipage}{0.5\textwidth}
\begin{align}
\un{(jl)}\un{(kj)}&=\un{(kj)}\un{(kl)}+\un{(kl)}\un{(jl)},&\label{reltri1}\\
\un{(kj)}\un{(jl)}&=\un{(jl)}\un{(kl)}+\un{(kl)}\un{(kj)},&\label{reltri2}\\
\uun{(jl)}\uun{(kj)}&=\uun{(kj)}\un{(kl)}+\un{(kl)}\uun{(jl)},&\label{reltri3}\\
\uun{(kj)}\uun{(jl)}&=\un{(jl)}\uun{(kl)}+\uun{(kl)}\un{(kj)}\label{reltri4},
\end{align}
\end{minipage}
\begin{minipage}{0.5\textwidth}
\begin{align}
\uun{(jl)}\un{(kj)}&=\un{(kj)}\uun{(kl)}+\uun{(kl)}\uun{(jl)},\label{reltri5}&\\
\uun{(kj)}\un{(jl)}&=\uun{(jl)}\uun{(kl)}+\uun{(kl)}\un{(kj)},\label{reltri6}&\\
\un{(jl)}\uun{(kj)}&=\uun{(kj)}\uun{(kl)}+\uun{(kl)}\un{(jl)},\label{reltri7}&\\
\uun{(kj)}\un{(jl)}&=\un{(jl)}\uun{(kl)}+\uun{(kl)}\uun{(kj)}.\label{reltri8}&
\end{align}
\end{minipage}
\end{lemma}
\begin{proof}
This follows from computing the braiding $\Psi$. The first relations (\ref{relduo1})-(\ref{relduo6}) are easy to see using the formulas for the module structure. Further, consider the circle
\[
\un{(kj)(kl)}\stackrel{\Psi}{\longmapsto} \un{(jl)(kj)}\stackrel{\Psi}{\longmapsto} \un{(kl)(jl)}\stackrel{\Psi}{\longmapsto} -\un{(kj)(kl)}
\]
from which we derive the relation (\ref{reltri1}). The relation (\ref{reltri2}) is derived from 
\[
\un{(kj)(jl)}\stackrel{\Psi}{\longmapsto} \un{(kl)(kj)}\stackrel{\Psi}{\longmapsto} -\un{(jl)(kl)}\stackrel{\Psi}{\longmapsto} -\un{(kj)(jl)}.
\]
The other relations (\ref{reltri3})-(\ref{reltri8}) are obtained by the same loops with all possible combinations of $\un{(~~)}$ and $\uun{(~~)}$.

The observation that $\Psi$ always maps elements of the basis of $Y_G\otimes Y_G$ consisting of tensor products of generators to basis elements (up to scalar $\pm$) can be used to verify that that these relations span $\ker(\ide+\Psi)$ (in fact, give a basis).
\end{proof}

Hence, Lemma~\ref{iquaddn} gives an explicit description in terms of generators and relations of the $D_n$-analogue of the algebra $\cE_{n}$ in \cite{FK}.

\begin{lemma}\label{idn}
The ideal $I_{D_{n}}$ is generated by the following commutator relations for $\lbrace i, j\rbrace \cap \lbrace k, l\rbrace=\emptyset$, and $k<j<l$:

\noindent\begin{minipage}{0.5\textwidth}
\begin{align}
[\un{(ij)},\un{(kl)}]&=0,&\label{antirelduo1}\\
[\uun{(ij)},\uun{(kl)}]&=0,&\label{antirelduo2}\\
[\un{(ij)},\uun{(kl)}]&=0,&\label{antirelduo3}
\end{align}
\end{minipage}
\begin{minipage}{0.5\textwidth}
\begin{align}
[\un{(jl)},\un{(kj)}]&=[\un{(kj)},\un{(kl)}]+[\un{(kl)},\un{(jl)}],&\label{antireltri1}\\
[\uun{(jl)},\uun{(kj)}]&=[\uun{(kj)},\un{(kl)}]+[\un{(kl)},\uun{(jl)}],&\label{antireltri2}\\
[\uun{(jl)},\un{(kj)}]&=[\un{(kj)},\uun{(kl)}]+[\uun{(kl)},\uun{(jl)}],&\label{antireltri3}\\
[\un{(jl)},\uun{(kj)}]&=[\uun{(kj)},\uun{(kl)}]+[\uun{(kl)},\un{(jl)}].&\label{antireltri4}
\end{align}
\end{minipage}
\end{lemma}
\begin{proof}
This in proved by intersecting $I^{\op{quad}}_G$ with $\Lambda Y_G$ using the basis from Lemma \ref{iquaddn}. The relation (\ref{relduo1})-(\ref{relduo3}) are anti-symmetric hence give (\ref{antirelduo1})-(\ref{antirelduo3}). No linear combinations of the relations (\ref{relduo4})-(\ref{relduo6}) give antisymmetric relations as these relations are symmetric. The relations (\ref{antireltri1})-(\ref{antireltri4}) are obtained as differences of the corresponding pairs of relations among (\ref{reltri1})-(\ref{reltri8}).
\end{proof}

\subsection{The \texorpdfstring{$B_n$}{Bn}-Case}\label{bncase}

We will now give an explicit presentation of $U(\fr{yb}_G)$ in the $B_n$-case. The reflections in $B_n = C_2^{n}\rtimes S_{n}$ are in one-to-one correspondence with the positive roots for $B_n$ which are $e_i$ in addition to $e_i\pm e_j$ $(1\leq i<j\leq n)$ which also appear in the $D_n$-case. We have $s_i=s_{e_i}$ for $i=1,\ldots,n$, which is
\begin{align}
s_i(e_k)&=\begin{cases} e_k,& k\neq i,\\-e_k,& k=i \end{cases}.
\end{align}
We denote the generator of $Y_G$ corresponding to $s_k$ by $r_k$. To describe the action of $B_n$ on $Y_{B_n}$ note the structure is the same on generators $\un{(ij)}$ and $\uun{(ij)}$ as in (\ref{dnmodulestructure}). In addition, we have

\noindent\begin{minipage}{0.5\textwidth}
\begin{align}
s_k\triangleright \un{(ij)}
&=\begin{cases}\un{(ij)},&\text{ if } k\neq i,j,\\
\uun{(ij)},&\text{ else},
\end{cases}\\
s_k\triangleright \uun{(ij)}
&=\begin{cases}\uun{(ij)},&\text{ if } k\neq i,j,\\
\un{(ij)},&\text{ else},
\end{cases}
\end{align}
\end{minipage}
\begin{minipage}{0.5\textwidth}
\begin{align}
(ij)\triangleright r_k
&=\begin{cases}r_k,&\text{ if } k\neq i,j,\\
r_j,&\text{ if } k=i,\\
r_i,&\text{ if } k=j.
\end{cases}\\
s_is_j(ij)\triangleright r_k
&=\begin{cases}r_k,&\text{ if } k\neq i,j,\\
r_j,&\text{ if } k=i,\\
r_i,&\text{ if } k=j.
\end{cases}
\end{align}
\end{minipage}

\begin{lemma}\label{iquadbn}
The ideal $I^{\op{quad}}_{B_n}$ is generated by the relations (\ref{relduo1})-(\ref{reltri8}) plus the additional relations for $k\neq i,j$

\noindent\begin{minipage}{0.5\textwidth}
\begin{align}
[r_k,\un{(ij)}]&=0,&\label{bnrel1}\\
[r_k,\uun{(ij)}]&=0,&\label{bnrel2}\\
[r_j,r_k]&=0,
\end{align}
\end{minipage}
\begin{minipage}{0.5\textwidth}
\begin{align}
r_i\un{(ij)}-\uun{(ij)}r_i+r_j\uun{(ij)}-\un{(ij)}r_j&=0,&\label{bnrel3}\\
r_j\un{(ij)}-\uun{(ij)}r_j+r_i\uun{(ij)}-\un{(ij)}r_i&=0.&\label{bnrel4}
\end{align}
\end{minipage}
\end{lemma}
\begin{proof}
First note that the subspace of $Y_{B_n}\otimes Y_{B_n}$ spanned by the $\un{(ij)}$, $\uun{(ij)}$ is closed under $\Psi$ and isomorphic to $Y_{D_{n}}\otimes Y_{D_{n}}$. All further relations arise by considering $\Psi$ applied to $r_i\otimes \un{(ij)}$, $r_i\otimes \uun{(ij)}$ or $r_i\otimes r_j$.
\end{proof}

\begin{lemma}\label{ibn}
The ideal $I_{B_n}$ is generated by the commutator relations (\ref{antirelduo1})-(\ref{antireltri4}), together with the relations (\ref{bnrel1})-(\ref{bnrel2}), and
\begin{align}
[r_i+r_j,\un{(ij)}+\uun{(ij)}]&=0.\label{bnrel5}&
\end{align}
\end{lemma}
\begin{proof}
This is easy to see using Lemma \ref{idn} and \ref{iquadbn}.
\end{proof}

The explicit presentation of $U(\fr{yb}_{D_{n}})$ and $U(\fr{yb}_{B_{n}})$ will enable us to show that these algebras are Koszul in Section \ref{sect3} by finding a PBW basis for the quadratic dual.

%%%%%%%%%%%%%%%%%%%%%%%%%%%%%%%%%%%%%%%%%%%%%%%%%%%%%%%%%%%%%%%%%%%%%%%%%%%%%%%%%%%%%%%%%%%%%%%%%%%%%%%

\section{Koszulness of BEER-algebras}\label{sect3}

\subsection{The Quadratic Dual}

The quadratic dual $A^!$ of a quadratic algebra $A=T(V)/{\langle R\rangle}$ is defined as $T(V^*)/{\langle R^\perp\rangle}$ using the orthogonal complement $R^\perp$ of the relations $R$ in $V^*\otimes V^*$. Returning to the algebra $U(\fr{yb}_G)$, we have that $S^2(Y_G) \leq I_G^{\perp}$. Hence, the quadratic dual $U(\fr{yb}_G)^!$ is antisymmetric and therefore a finite-dimensional PI algebra.

In the following, we will construct a PBW basis for the quadratic dual $U(\fr{yb}_G)^!$ which implies that $U(\fr{yb}_G)^!$ is Koszul, and hence $U(\fr{yb}_G)$ is Koszul as well by standard theory of Koszul algebras \cite{Pri,BGS}.

We denote the generator dual to $\un{(ij)}$ of $U(\fr{yb}_G)$ by $\ov{(ij)}\in U(\fr{yb}_G)^!$ for $G=A_{n-1}, D_{n}, B_n$, and the dual elements to $\uun{(ij)}$ by $\oov{(ij)}$. Using the result of \cite{BEER}, we see that $U(\fr{tr}_n)^!$ is the anti-symmetric algebra generated by the relations
\begin{align}
\ov{(ij)(jk)}&=\ov{(ik)(jk)}=\ov{(jk)(ij)}=\ov{(ij)(ik)}.&\label{quadrel1}
\end{align}
for all $i<j<k$.

\subsection{Koszulness for the \texorpdfstring{$D_n$}{Dn}-case}

\begin{lemma}\label{quadraticdualdn}
The quadratic dual $U(\fr{yb}_{D_n})^!$ is the antisymmetric algebra generated by $\ov{(ij)}$, $\oov{(ij)}$ for $1\leq i<j\leq n$, which are, for $k<j<l$, subject to the relations (\ref{quadrel1}) and
\begin{align}
\oov{(jl)}\oov{(kl)}&=-\ov{(kj)}\oov{(jl)}=-\ov{(kj)}\oov{(kl)},&\label{quadrel2}\\
\oov{(kl)}\oov{(kj)}&=~~~\ov{(jl)}\oov{(kl)}=-\ov{(jl)}\oov{(kj)},&\label{quadrel3}\\
\oov{(kj)}\oov{(jl)}&=~~~\ov{(kl)}\oov{(jl)}=-\ov{(kl)}\oov{(kj)}.&\label{quadrel4}\\
\ov{(ij)}\oov{(ij)}&=~~~0.&\label{quadrel5}
\end{align}
\end{lemma}
\begin{proof}
This follows by computing the orthogonal complement of $\ker(\Psi+\ide_{Y_{D_n}\otimes Y_{D_n}})$ using Lemma \ref{iquaddn}.
\end{proof}

\begin{proposition}\label{basisdn}
A vector space basis $U(\fr{yb}_{D_n})^!$ is given by products of monomials of the form
\begin{align}
&\ov{(j,i_1)}\ov{(j,i_2)}\ldots\ov{(j,i_k)},&\label{reduceddn1}\\
&\ov{(j,i_1)}\ov{(j,i_2)}\ldots \ov{(j,i_{p-1})}\ov{(j,i_{p+1})}\ldots\ov{(j,i_k)}\oov{(j,i_p)},&\forall p=2,\ldots,k,\label{reduceddn2}\\
&\ov{(i_1,i_2)}\ov{(i_1,i_3)}\ldots\ov{(i_1,i_k)}\oov{(j,i_1)}.&\label{reduceddn3}
\end{align}
for \emph{disjoint} index sets $\lbrace j<i_1<i_2<\ldots<i_k\rbrace$. The monomials of this form are arranged according to the value of the respective smallest elements $j$ to give basis elements. We refer to monomials of the form (\ref{reduceddn1})-(\ref{reduceddn3}) as \emph{reduced monomials}.
\end{proposition}

\begin{example}
If $n=4$, the monomials

\noindent\begin{minipage}{0.20\linewidth}
\begin{align*}
&\ov{(12)}\ov{(13)}\ov{(14)},&\\
&\ov{(12)}\ov{(14)}\oov{(13)},&\\
&\ov{(12)}\ov{(13)}\oov{(14)},&\\
&\ov{(23)}\ov{(24)}\oov{(12)},&\\
&\ov{(23)}\oov{(12)},&\\
&\ov{(24)}\oov{(12)},&\\
&\ov{(34)}\oov{(23)},&
\end{align*}
\end{minipage}
\begin{minipage}{0.15\linewidth}
\begin{align*}
&\ov{(12)}\ov{(34)},&\\
&\ov{(13)}\ov{(24)},&\\
&\ov{(14)}\ov{(23)},&\\
&\ov{(12)}\oov{(34)},&\\
&\ov{(13)}\oov{(24)},&\\
&\ov{(14)}\oov{(23)},&
\end{align*}
\end{minipage}
\begin{minipage}{0.15\linewidth}
\begin{align*}
&\oov{(12)}\ov{(34)},&\\
&\oov{(13)}\ov{(24)},&\\
&\oov{(14)}\ov{(23)},&\\
&\oov{(12)}\oov{(34)},&\\
&\oov{(13)}\oov{(24)},&\\
&\oov{(14)}\oov{(23)},&
\end{align*}
\end{minipage}
\begin{minipage}{0.15\linewidth}
\begin{align*}
&\ov{(12)}\ov{(13)},&\\
&\ov{(12)}\ov{(14)},&\\
&\ov{(12)}\oov{(13)},&\\
&\ov{(12)}\oov{(14)},&\\
&\ov{(23)}\ov{(24)},&\\
&\ov{(23)}\oov{(24)},&
\end{align*}
\end{minipage}
\begin{minipage}{0.15\linewidth}
\begin{align*}
&\ov{(12)},&\\
&\ov{(13)},&\\
&\ov{(14)},&\\
&\ov{(23)},&\\
&\ov{(24)},&\\
&\ov{(34)},&
\end{align*}
\end{minipage}
\begin{minipage}{0.15\linewidth}
\begin{align*}
&\oov{(12)},&\\
&\oov{(13)},&\\
&\oov{(14)},&\\
&\oov{(23)},&\\
&\oov{(24)},&\\
&\oov{(34)},&
\end{align*}
\end{minipage}
give a full list of all monomials in $U(\fr{yb}_{D_4})^!$. Hence the Hilbert-Poincar\'e polynomial of this $38$-dimensional algebra is
\begin{align*}
P^!_{D_4}(t)&=1+12t+21t^2+4t^3.&
\end{align*}
From this we can compute the Hilbert-Poincar\'e polynomial of $U(\fr{yb}_{D_4})$ as $P_{D_4}(t)=(P^!_{D_4}(-t))^{-1}$, giving
\begin{align*}
P_{D_4}(t)=1+12t+123t^2+1228t^3+12201t^4+121116t^5+O(t^6).
\end{align*}
\end{example}

To proof Proposition \ref{basisdn}, we provide an algorithm to transform an arbitrary monomial in the generators into a reduced one. Given a monomial
\begin{align*}
M&=\prod_{i=1}^k\ov{(a_i,b_i)}\prod_{j=1}^l\oov{(c_j,d_j)},&
\end{align*}
we consider the graph $\Gamma$ with vertices $\lbrace a_1,b_1\ldots,a_k,b_k,c_1,d_1,\ldots,c_l,d_l  \rbrace$ and edges $a_i\leftrightarrow b_i$, $c_j\leftrightarrow d_j$. The monomial $M$ can be decomposed into a product of monomials for each connected component of the graph $\Gamma$, using the antisymmetric relations. For the following algorithm we assume without loss of generality that $\Gamma$ is connected and the set of vertices is $\lbrace 1,\ldots,n \rbrace$.

\begin{algorithm}\label{dnalgorithm}$~$
\begin{enumerate}
\item[(1)] If $M$ is a product of generators $\ov{(a_i,b_i)}$ only (i.e. $l=0$), then $M=0$ or $M=\pm \ov{(1,2)}\ov{(1,3)}\ldots \ov{(1,n)}$ using relation (\ref{quadrel1}), cf. \cite{BEER}. Hence $M$ is a scalar multiple of a monomial of the form (\ref{reduceddn1}).
\item[(2)] If $l\neq 0$, so some $\oov{(c_i,d_i)}$ occurs, we can use antisymmetry, the relations (\ref{quadrel2})-(\ref{quadrel4}), and step (1) to bring $M$ into the form
\begin{align*}
M&=\pm \ov{(a,b_1)}\ov{(a,b_2)}\ldots\ov{(a,b_m)}\oov{(c,d)},&
\end{align*}
where $a<b_1<\ldots<b_m$ or $M=0$. If $\lbrace c,d \rbrace\subset \lbrace{a,b_1,\ldots,b_m}\rbrace$, then $M=0$ (from (\ref{quadrel5})). 
Otherwise, we distinguish the following cases:
\begin{enumerate}
\item[(2a)] If $c=a=1$, then
\begin{align*}
M&=\pm \ov{(1,2)}\ov{(1,3)}\ldots \ov{(1,d-1)}\ov{(1,d+1)}\ldots \ov{(1,n)}\oov{(1,d)},&
\end{align*}
and hence has the form $\pm$(\ref{reduceddn2}).
\item[(2b)] If $c=1$, $d=a$, then
\begin{align*}
M&=\pm\ov{(2,3)}\ov{(2,4)}\ldots\ov{(2,n)}\oov{(1,2)},&
\end{align*}
so it has the form $\pm$(\ref{reduceddn3}).
\item[(2c)] If $c=b_i$ for some $i$, then $a=1$ and
\begin{align*}
M&=\pm\ov{(1,2)}\ov{(1,2)}\ldots\ov{(1,d-1)}\ov{(1,d+1)}\ldots\ov{(1,n)}\oov{(1,d)},&
\end{align*}
using the relation $\ov{(1,b_i)}\oov{(b_i,d)}=\ov{(1,b_j)}\oov{(1,d)}$ from (\ref{quadrel4}). Hence $M$ has the form $\pm$(\ref{reduceddn2}).
\item[(2d)] If $c=1$ and $d=b_i$ for some $i$, then $a=2$, and
\begin{align*}
M&=\pm\ov{(2,3)}\ov{(2,4)}\ldots\ov{(2,n)}\oov{(1,2)},&
\end{align*}
using $\ov{(2,b_i)}\oov{(1,b_i)}=-\ov{(2,b_i)}\oov{(1,2)}$ which follows from (\ref{quadrel3}). This implies that $M$ has the form $\pm$(\ref{reduceddn3}).
\item[(2e)] If $c\neq 1$ and $d=b_i$, then
\begin{align*}
M&=\pm \ov{(1,2)}\ov{(1,2)}\ldots\ov{(1,d-1)}\ov{(1,d+1)}\ldots\ov{(1,n)}\oov{(1,d)},&
\end{align*}
using $\ov{(1,b_i)}\oov{(c,b_i)}=-\ov{(1,b_i)}\oov{(1,c)}$ from (\ref{quadrel4}). Hence $M$ is of the form $\pm$(\ref{reduceddn2}).
\end{enumerate}
\end{enumerate}
\end{algorithm}

The algorithm proves Proposition \ref{basisdn}. Next, we will show that the basis from Proposition \ref{basisdn} is a PBW basis. That is, there exists a choice of a total order on the set of generators $r_1,\ldots,r_k$ such that the basis is of the following form:  We denote by $\cT$ the set of pairs of indices $(i,j)$ such that $r_ir_j$ cannot be expressed as a linear combination of elements in (lexicographically) smaller order. In this case, the PBW basis is given by
\begin{align*}
\cB&=\lbrace r_{i_1}\ldots r_{i_l}\mid (i_k,i_{k+1})\in \cT, \forall k=1,\ldots, l-1\rbrace.&
\end{align*}
Note that for a general total order on the generators, the set $\cB$ is spanning, but not necessarily linearly independent.

We introduce the total ordering on the generators $\ov{(ij)}$, $\oov{(ij)}$ by $\ov{(ij)}<\oov{(lk)}$ and ordering the $\ov{(ij)}$ (and $\oov{(ij)}$) lexicographically.

\begin{lemma}\label{tfordn}
For the above order on the generators of $U(\fr{yb}_{D_n})$, the corresponding set $\cT$ consists of the pairs

\noindent\begin{minipage}{0.5\textwidth}
\begin{align*}
&\ov{(ab)}\ov{(cd)}, &\\
&\ov{(ab)}\oov{(cd)}, &\\
&\oov{(ab)}\oov{(cd)},
\end{align*}
\end{minipage}
\begin{minipage}{0.5\textwidth}
\begin{align*}
&\ov{(kj)}\ov{(kl)}, &\\
&\ov{(kj)}\oov{(kl)},&\\
&\ov{(jl)}\oov{(kj)}.&
\end{align*}
\end{minipage}
for $a,b,c,d$ pairwise distinct ($a<b$, $c<d$) and $k<j<l$.
\end{lemma}
\begin{proof}
For pairs with four distinct indices, the only relations are the anti-commutativity relations so precisely the given pairs in the first column are in $\cT$ as they have the lexicographically smallest order .

For pairs which share one index, consider the relations (\ref{quadrel1})-(\ref{quadrel4}). All possible pairs appear in these relations. Hence, the ones with the smallest lexicographic order are precisely the ones in $\cT$. All pairs with the same index set are zero.
\end{proof}

\begin{theorem}
The basis of Proposition \ref{basisdn} is the PBW basis with respect to the above order on the generators. In particular, $U(\fr{yb}_{D_n})$ is a Koszul algebra.
\end{theorem}
\begin{proof}
This can be verified by direct checking using the description of the set $\cT$ from Lemma \ref{tfordn}.
\end{proof}

\subsection{Koszulness for the \texorpdfstring{$B_n$}{Bn}-case}

We will now extend the Koszulness result to the $B_n$-case. For this, recall the presentation of $U(\fr{yb}_{B_n})$ from Section \ref{bncase}.

\begin{lemma}
The quadratic dual $U(\fr{yb}_{B_n})^!$ is the antisymmetric algebra generated by $\ov{(ij)}$, $\oov{(ij)}$ for $1\leq i<j\leq n$, and $r^i$, for $i=1, \ldots, n$, which are subject to the relations (\ref{quadrel1})-(\ref{quadrel5}) and 
\begin{align}\label{quadreln6}
r^i\ov{(ij)}&=r^i\oov{(ij)}=r^j\ov{(ij)}=r^j\oov{(ij)},& j<i_1<\ldots<i_k.
\end{align}
\end{lemma}
\begin{proof}
This follows computing the orthogonal complement of $I_{B_n}$ using the basis from Lemma \ref{iquadbn}, cf. Lemma \ref{quadraticdualdn}.
\end{proof}

\begin{proposition}\label{basisbn}
A vector space basis $U(\fr{yb}_{B_n})^!$ is given by products of monomials of the form (\ref{reduceddn1})-(\ref{reduceddn3}) as well as monomials of the form
\begin{align}
&\ov{(j,i_1)}\ov{(j,i_2)}\ldots\ov{(j,i_k)}r_j.&\label{reducedbn}
\end{align}
\end{proposition}

We can adjust Algorithm \ref{dnalgorithm} to the $B_n$-case to prove Proposition \ref{basisbn}.

\begin{algorithm}\label{bnalgorithm}$~$
\begin{enumerate}
\item[(1)] Given a monomial $M$ in the generators for $U(\fr{yb}_{B_n})$ such that $\Gamma$ is connected on the set $\lbrace 1,\ldots, n\rbrace$, use antisymmetry relations to bring it into the form
\begin{align*}
M&=\pm\prod_i^p\ov{(a_i,b_i)}\prod_j^q\oov{(c_j,d_j)}\prod_k^sr^k.&
\end{align*}
\item[(2)] Use Algorithm \ref{dnalgorithm} to bring $M'=\prod_i^p\ov{(a_i,b_i)}\prod_j^q\oov{(c_j,d_j)}$ into the form of a basis element of $U(\fr{yb}_{D_n})$.
\item[(3)] Using the relations (\ref{quadreln6}), $M=0$ unless $s=0,1$. If $s=0$, we are done as $M=M'$ is a basis element of the form (\ref{reduceddn1})-(\ref{reduceddn3}) up to sign.
\item[(4)] If $s=1$, use the relations (\ref{quadreln6}) interchange all generators $\oov{(ij)}$ with the corresponding generators $\ov{(ij)}$. Using step (2), we can then bring $M$ into the form
\begin{align*}
M&=\pm \ov{(1,2)}\ov{(1,3)}\ldots\ov{(1,n)}r^k.
\end{align*}
Using (\ref{quadreln6}) again, we can replace $r^k$ by $r^1$. Hence, $M$ has the form $\pm$(\ref{reducedbn}) in this case.
\end{enumerate}
\end{algorithm}

We can give a total order on the generators of $U(\fr{yb}_{B_n})$ by ordering the $\ov{(ij)}$ (and $\oov{(ij)}$) lexicographically, setting $\ov{(a,b)}<\oov{(c,d)}$ for all indices, and $r^k>\oov{(c,d)}$ for all $k, c, d$. Where the $r^k$ are ordered $r^1<r^2< \ldots < r^n$.

Using this total order the corresponding set $\cT$ consists of the pairs from Lemma \ref{tfordn} and the additional pairs
\begin{align*}
\ov{(ij)}r^k,&\qquad \oov{(ij)}r^k,\qquad \ov{(ij)}r^i,&\text{for pairwise distinct indices $i,j, k$}
\end{align*}

\begin{theorem}
The basis of Proposition \ref{basisbn} is the PBW basis with respect to the above order on the generators. In particular, $U(\fr{yb}_{B_n})$ is a Koszul algebra.
\end{theorem}

\begin{example}
We can now compute the Hilbert polynomial for $U(\fr{yb}_{B_n})$ and its quadratic dual using the PBW basis. For $n=4$, it is
\begin{align*}
P^!_{B_4}(t)&=1+72t+51t^2+5t^3,&\\
P_{B_4}(t)&=1+72t+5133t^2+365909t^3+26084025t^4+1859414106 t^5=O(t^6).&
\end{align*}
\end{example}

\subsection{Consequences}

Koszulness of the algebras $U(\fr{yb}_G)$ enables us to compute the extension algebras explicitly. If is well known that for a Koszul algebra $U$,
\begin{align*}
\Ext^*_U(k,k)=\bigoplus_{l,m\geq o}\Ext^{l,m}_U(k,k)=\bigoplus_{l\geq 0}\Ext^l_U(k,k)\cong U^!,
\end{align*}
which is thus generated by $\Ext^1_U(k,k)$. A free resolution for $U$ is given by the \emph{Koszul complex} \cite{Pri,BGS}. The module categories can be related by \emph{Koszul duality} \cite[Theorem 1.2.6]{BGS} giving an equivalence of triangulated categories between the bounded derived categories of finite-dimensional graded modules over $U$ and $U^!$. The latter is a finite-dimensional PI algebra.

%%%%%%%%%%%%%%%%%%%%%%%%%%%%%%%%%%%%%%%%%%%%%%%%%%%%%%%%%%%%%%%%%%%%%%%%%%%%%%%%%%%%%%%%%%%%%%%%%%%%%%%%

\section{Embeddings of Heisenberg Doubles of BEER-Algebras and Functors of Rational Cherednik Algebra Representations}\label{sect2}

In this section, we use the notion of perfect subquotient from \cite{BB} to obtain maps between Heisenberg doubles of the BEER-algebras $U(\fr{yb}_G)$ for different $G$ of the classical series $A_n$, $B_n$, $D_n$. These can be used to construct functors between categories of representations of rational Cherednik algebras. To ensure irreducibility, we restrict to $n\geq 3$ in the $D_n$-case for the applications to rational Cherednik algebras.

\subsection{Maps of Heisenberg Doubles of Different Series}\label{differentseriesmaps}

Note that $A_{n-1}=S_n$ is both a subgroups (and quotient) $D_n$ (and $B_n$) of  using the group homomorphisms
\begin{align}
S_n&\stackrel{j_A^D}{\hooklongrightarrow}D_n\stackrel{p_{A}^D}{\twoheadlongrightarrow}S_n,&\\
S_n&\stackrel{j_A^B}{\hooklongrightarrow}B_n\stackrel{p_{A}^B}{\twoheadlongrightarrow}S_n.&
\end{align}
These morphisms supply induction functors
\begin{align}
\begin{split}\Ind_A^D\colon \leftexpsub{S_n}{S_n}{\mathcal{YD}}&\longrightarrow \leftexpsub{D_n}{D_n}{\mathcal{YD}},\\
(V,\triangleright, \delta)&\longmapsto (V,\triangleright(p_A^D\otimes \ide_V),(j_A^D\otimes\ide_V)\delta), \end{split}\\
\begin{split}\Ind_A^B\colon \leftexpsub{S_n}{S_n}{\mathcal{YD}}&\longrightarrow \leftexpsub{B_n}{B_n}{\mathcal{YD}},\\
(V,\triangleright, \delta)&\longmapsto (V,\triangleright(p_A^B\otimes \ide_V),(j_A^B\otimes\ide_V)\delta).\end{split}
\end{align}
Via reconstruction theory, this gives rise to maps
\begin{align*}
d_A^D\colon \Drin(S_n)&\longrightarrow \Drin(D_n),&\\
d_A^B\colon \Drin(S_n)&\longrightarrow \Drin(B_n).&
\end{align*}
Comparing generators gives rise to maps
\begin{align}
\Ind_A^D Y_{S_n}&\stackrel{\iota_A^D}{\hooklongrightarrow} Y_{D_n}\stackrel{\pi_A^D}{\twoheadlongrightarrow} \Ind_A^D Y_{S_n}.&\\
\Ind_A^B Y_{S_n}&\stackrel{\iota_A^B}{\hooklongrightarrow} Y_{B_n}\stackrel{\pi_A^B}{\twoheadlongrightarrow}\Ind_A^BY_{S_n},&
\end{align}
of YD-modules over $D_n$ (respectively, $B_n$).
These pairs of maps are \emph{perfect subquotients} in the terminology of \cite{BB}. Such perfect subquotients are used to induce maps between the corresponding braided Heisenberg doubles of the braided Hopf algebras $U(\fr{yb}_G)$ (as YD-modules over $D_n$, respectively $B_n$).

\begin{lemma}\label{subquotients}
The pair $(\iota_A^D, \pi_A^D)$ is a perfect subquotient of YD-modules over $D_n$, and $(\iota_A^B,\pi^B_A)$ gives a perfect subquotients of YD-modules over $B_n$.
\end{lemma}
\begin{proof}
Consider the pair $(\iota_A^D, \pi_A^D)$. To show it is a subquotient, we need $(\ide\otimes \pi_A^D)\delta_{D_n}\iota_A^D=(j_A^D\otimes \ide)\delta_{S_n}$, where $\delta_{G}$ denotes the $G$-coaction on $Y_{G}$. But this is clear as $\delta_{D_n}(\un{(ij)})=(ij)\otimes \un{(ij)}=\delta_{S_n}(\un{(ij)})$.

Next, the check the subquotients are \emph{perfect}, we need for the maximal triangular ideals\footnote{For a YD-module $(V,\delta)$ over $H$, the maximal triangular ideal $I(V,\delta)$ is the maximal homogeoneous ideals in degree $\geq 2$ which is also a coideal.} (cf. \cite{BB}) that
\begin{align*}
I(\Ind_A^D Y_{S_n}, (j_A^D\otimes \ide)\delta_{S_n})&=T(\iota_A^D)^{-1}I(Y_{D_n},\delta_{D_n}).&
\end{align*}
The right hand side can be identified with $T(Y_{S_n})\cap I(Y_{D_n},\delta_{D_n})$. Then it is easy to see using Lemma \ref{iquaddn} that the relations among the generators $\un{(ij)}$ are the same in $Y_{D_n}$ as in $Y_{S_n}$ which implies the equality. The proof for $Y_{B_n}$ is analogous.
\end{proof}

We remark that the braided Heisenberg double from \cite{BB} is defined on $U(\fr{yb}_G^*)\otimes kG\otimes U(\fr{yb}_G)$, and the pairing between $U(\fr{yb}_G^*)$ and $U(\fr{yb}_G)$ is \emph{not} perfect. Note that as the Drinfeld doubles are quasitriangular, we have a map $\Drin(G)\to \leftexpsub{\Drin(G)}{\Drin(G)}{\mathcal{YD}}$ using the universal $R$-matrix $\sum_{g\in G}{\delta_g\otimes g}$. We will use this to compute the Heisenberg double of $U(\fr{yb}_G)$ over $\Drin(G)$. We will also compute the Heisenberg double over $\mC G$. The version over $\mC G$ is for the purpose of Section \ref{heisenbergmaps}, while the version over $\Drin(G)$ is relevant for \cite{Lau2}.

We use the notations $\ov{(ij)}$, $\oov{(ij)}$ for generators of $U(\fr{yb}_{G_n})$ for convenience even though this notation was already used for the quadratic dual.

\begin{corollary}\label{heisenbergmaps}
There exist injective algebra morphisms
\begin{align*}
\phi_A^D\colon \Heis_{\Drin(S_{n})}(U(\fr{tr}_{n}))&\hooklongrightarrow \Heis_{\Drin(D_n)}(U(\fr{yb}_{D_n})),&\\
\phi_A^B\colon \Heis_{\Drin(S_{n})}(U(\fr{tr}_{n}))&\hooklongrightarrow \Heis_{\Drin(B_n)}(U(\fr{yb}_{B_n})),&\\
\phi_A^D\colon \Heis_{\mC S_{n}}(U(\fr{tr}_{n}))&\hooklongrightarrow \Heis_{\mC D_n}(U(\fr{yb}_{D_n})),&\\
\phi_A^B\colon \Heis_{\mC S_{n}}(U(\fr{tr}_{n}))&\hooklongrightarrow \Heis_{\mC B_n}(U(\fr{yb}_{B_n})).&
\end{align*}
\end{corollary}
\begin{proof}
The existence of a morphism of algebra $\Heis_{\Drin(D_n)}(U(\fr{tr}_{n})) \rightarrow \Heis_{\Drin(D_n)}(U(\fr{yb}_{D_n}))$ follows by applying the general theory of perfect subquotients in \cite{BB} to the ones constructed in Lemma \ref{subquotients}. The algebra morphism stated is obtained by pre-composing with the algebra morphism $\ide_{U(\fr{tr}_{n})}\otimes d_A^D\otimes \ide_{U(\fr{tr}_{n}^*)}$. The construction for $B_n$ is the same using $d_A^B$. Concretely, the morphism is given by mapping $\un{(ij)} \in U(\fr{tr}_{n})$ to $\un{(ij)} \in U(\fr{yb}_{D_n})$ (and mapping the dual generators $\ov{(ij)} \in U(\fr{tr}_n^*)$ to $\ov{(ij)}\in U(\fr{yb}_{D_n}^*)$), the group element $(ij)\in S_n$ maps to $(ij)\in D_n$, and the function $\delta_{(ij)} \in \mC[S_{n}]$ maps to $\delta_{(ij)}={p_A^D}^*\delta_{(ij)} \in \mC[D_n]$. Hence we see that $\phi_A^D$ is injective. The proof for $B_n$ in place of $D_n$ is again analogous. 

Both morphisms $\phi_A^D$ and $\phi_A^B$ admit versions over $\mC G$ if we regard $Y_G$ as YD-modules over $G$ rather than its Drinfeld double. This first gives a morphism $\Heis_{\mC D_n}(U(\fr{tr}_{n})) \rightarrow \Heis_{\mC D_n}(U(\fr{yb}_{D_n}))$ which will give the desired morphism by pre-composition with $\ide_{U(\fr{tr}_{n})}\otimes j_A^D\otimes \ide_{U(\fr{tr}_{n}^*)}$. 
\end{proof}

We can also use the subquotient result \ref{subquotients} to obtain morphisms of Hopf algebras between the corresponding braided Drinfeld doubles (or \emph{double bosonizations} in \cite{Maj2}). These will be relevant in \cite{Lau2} in order to provide categorical actions of the monoidal category of modules over the braided Drinfeld double of type $A_{n-1}$ on modules over rational Cherednik algebras of different types.
\begin{align*}
\phi_A^D\colon \Drin_{\Drin(S_{n})}(U(\fr{tr}_{n}))&\hooklongrightarrow \Drin_{\Drin(D_n)}(U(\fr{yb}_{D_n})),&\\
\phi_A^B\colon \Drin_{\Drin(S_{n})}(U(\fr{tr}_{n}))&\hooklongrightarrow \Drin_{\Drin(B_n)}(U(\fr{yb}_{B_n})).&
\end{align*}
Note that the unlike the braided Heisenberg double, the braided Drinfeld double of $U(\fr{yb}_G)$ \emph{cannot} be computed over $\mC G$.

\subsection{Application to Rational Cherednik Algebras}\label{applications}

Note first that in the cases $A_n$, $D_n$ and $B_n$, all reflections are conjugate. Therefore, the parameter $c=(c_s)_{s\in \cS}$ is just a constant. In the following, we can also consider different constants for the functors obtained in the $t=0$ case as all rational Cherednik algebras embed into the same Heisenberg double.

The maps from Lemma~\ref{subquotients}
\begin{align*}
\phi_A^D\colon \Heis_{\mC S_{n}}(U(\fr{tr}_{n}))&\hooklongrightarrow \Heis_{\mC D_n}(U(\fr{yb}_{D_n})),&\\
\phi_A^B\colon \Heis_{\mC S_{n}}(U(\fr{tr}_{n}))&\hooklongrightarrow \Heis_{\mC B_n}(U(\fr{yb}_{B_n})),&
\end{align*}
can be combined with the maps from \cite{BB},
\begin{align*}
M_c(G)\colon H_{0,c}(G)&\longrightarrow \Heis_{\mC G}(U(\fr{yb}_G)).&
\end{align*}
These induce functors
\begin{align*}
\op{HInd}_{A}^{D}\colon \lmod{H_{0,c}(S_n)}&\longrightarrow \lmod{H_{0,c}(D_n)},&\\
\op{HInd}_A^{B}\colon \lmod{H_{0,c}(S_n)}&\longrightarrow \lmod{H_{0,c}(B_n)}.&
\end{align*}
The functors are given by
\begin{align*}
\op{HInd}_{A}^{D}&=M_c(D_n)^*{\phi_A^D}_*M_c(S_n)_*,&\\
\op{HInd}_A^{B}&=M_c(B_n)^*{\phi_A^B}_*M_c(S_n)_*,
\end{align*}
where $(-)_*$ is the extension functor $\Heis_{\mC G}(U(\fr{yb}_G))\otimes_{H_{t,c}(S_n)}(-)$ (pushforward), and $(-)^*$ is the restriction functor (pullback). Note that the functor $\op{HInd}_A^D$ is given by the $H_{0,c}(D_n)$-$H_{0,c}(S_n)$-bimodule $\Heis_{\mC D_n}(U(\fr{yb}_{D_n}))$ (and $\op{HInd}_A^B$ is given by the bimodule using $B_n$).

We can also consider the functors
\begin{align*}
\op{HInd}_{D}^{A}\colon \lmod{H_{0,c}(D_n)}&\longrightarrow \lmod{H_{0,c}(S_n)},&\\
\op{HInd}_B^{A}\colon \lmod{H_{0,c}(B_n)}&\longrightarrow \lmod{H_{0,c}(S_n)}.&
\end{align*}
These functors are given by tensoring with $\Heis_{\mC D_n}(U(\fr{yb}_{D_n}))$ viewed as a $H_{0,c}(S_n)$-$H_{0,c}(D_n)$-bimodule for type $D_n$ (and using instead $B_n$ of $D_n$ for type $B$). By construction, the induction functors have natural right adjoints (using the Tensor-Hom adjunction) which we denote by $\op{HRes}_A^D, \op{HRes}_{D}^A$, (or using $B$ instead of $D$). For example, the functor $\op{HRes}_A^D \colon \lmod{H_{0,c}(D_n)} \longrightarrow \lmod{H_{0,c}(S_n)}$ is given by mapping an object $V$ to the $H_{0,c}(D_n)$-module $\Hom_{H_{0,c}(S_n)}(\Heis_{\mC D_n}(U(\fr{yb}_{D_n}),V)$. The module structure is given by $(a\triangleright f)(b)=f(a\triangleright b)$.

\begin{corollary}
All versions of the functors $\op{HInd}$ have right adjoints and hence preserve colimits, and in particular, are right exact.
\end{corollary}

It is possible to obtain versions of the above functors for the case $t\neq 0$ and for the restricted rational Cherednik algebras, again using the embeddings of \cite{BB} and the observations about subquotients. In the case $t\neq 0$, the bimodules are given by twisted Heisenberg doubles of $U(\fr{yb}_G)$. Caution about the parameters is however in order in these cases. For the restricted versions, the actual Nichols algebras $\cB(Y_G)$ are used instead of $U(\fr{yb}_G)$.

\subsection{Maps of Heisenberg Doubles of the Same Series}\label{sameseriesmaps}

We can also consider the map
\begin{align*}
\tau_n^{n+1}\colon\Heis_{\mC S_n}(U(\fr{tr}_n))&\longrightarrow \Heis_{\mC S_{n+1}}(U(\fr{tr}_{n+1})),&
\end{align*}
which maps $S_n\hookrightarrow S_{n+1}$ by acting on $\lbrace 1,\ldots, n\rbrace$, and maps $\un{(ij)} \mapsto \un{(ij)}$, $\ov{(ij)} \mapsto \ov{(ij)}$. Note that this map is an injective morphism of algebras. Similarly, such maps exist for types $B_n$, $D_n$. Again, it is also possible to give morphisms of Drinfeld doubles
\begin{align*}
\tau_n^{n+1}\colon\Drin_{\Drin(S_n)}(U(\fr{tr}_n))&\longrightarrow \Drin_{\Drin(S_{n+1})}(U(\fr{tr}_{n+1})).&
\end{align*}

As in Section \ref{applications}, we can combine the maps $\tau_n^{n+1}$ with the maps $M_c(G)$ to obtain functors
\begin{align*}
\op{HInd}_n^{n+1}\colon \lmod{H_{0,c}(G_n)}&\longrightarrow \lmod{H_{0,c}(G_{n+1})},&
\op{HInd}_{n+1}^{n}\colon \lmod{H_{0,c}(G_{n+1})}&\longrightarrow \lmod{H_{0,c}(G_{n})},&
\end{align*}
where $G_n$ is $S_n$, $B_n$ or $D_n$. The induction functor $\op{HInd}_n^{n+1}$ is defined by considering the algebra $\Heis_{\mC G_{n+1}}(U(\fr{yb}_{G_{n+1}}))$ as a $H_{0,c}(G_{n+1})$-$H_{0,c}(G_n)$-bimodule. The left $H_{0,c}(G_{n+1})$-action is given by restriction along the morphism $M_c(G_{n+1})$, where the right $H_{0,c}(G_n)$-structure is given by restriction along the composite $\tau_n^{n+1}M_c(G_n)$. Similarly, the functor $\op{HInd}_{n+1}^n$ is given by tensoring with the $H_{0,c}(G_{n})$-$H_{0,c}(G_{n+1})$-bimodule $\Heis_{\mC G_{n+1}}(U(\fr{yb}_{G_{n+1}}))$.

%All versions of functors $\op{HInd}$, $\op{HRes}$ have right adjoints by the cartesian closed structure of the categories of modules.

As before, we can define right adjoints, denoted by $\op{HRes_n^{n+1}}$, $\op{HRes}_{n+1}^n$ which imply that the induction functors are right exact.

\subsection{A More Complete Picture}

A more general method to obtain morphisms between the braided Heisenberg doubles of algebras $U(\fr{yb}_G)$ for finite Coxeter groups can be given by comparing Dynkin diagrams. If one Dynkin diagram embeds into another one, then the BEER-algebra of the smaller one will be perfect subquotient of the BEER-algebra of the group corresponding to the larger Dynkin diagram. Hence, maps between their braided Heisenberg doubles can be given, so that restriction and induction functors can be discussed as done above. For example, the diagram of type $A_{n-1}$ embeds into both the diagram of type $D_n$ and $B_n$ giving rise to the morphisms in Section \ref{differentseriesmaps}. There are also embeddings of the Dynkin diagrams of smaller rank into the Dynkin diagrams of larger rank for the classical series, giving the morphisms of braided Heisenberg doubles in Section \ref{sameseriesmaps}. Analogues of such morphisms can also be defined for types $B_n$ and $D_n$.

%%%%%%%%%%%%%%%%%%%%%%%%%%%%%%%%%%%%%%%%%%%%%%%%%%%%%%%%%%%%%%%%%%%%%%%%%%%%%%%%%%%%%%%%%%%%%%%%%%%%%%%%%%

%\addcontentsline{toc}{section}{References}
%\nocite{*}
\bibliography{biblio}
\bibliographystyle{amsalpha}%agsm or dcu
\phantomsection

\end{document}